\documentclass[10pt]{amsart}
\usepackage{latexsym}
\usepackage{amsfonts}

\usepackage{amsmath,amsthm,amssymb}

\title[Random length-spectrum rigidity]{Random length-spectrum rigidity for free groups}

\author[I.~Kapovich]{Ilya Kapovich}

\address{\tt Department of Mathematics, University of Illinois at
  Urbana-Champaign, 1409 West Green Street, Urbana, IL 61801, USA
  \newline http://www.math.uiuc.edu/\~{}kapovich/} \email{\tt
  kapovich@math.uiuc.edu}

\newtheorem{thm}{Theorem}[section] \newtheorem{lem}[thm]{Lemma}
\newtheorem{cor}[thm]{Corollary} 
\newtheorem{prop}[thm]{Proposition} \theoremstyle{definition}
\newtheorem{defn}[thm]{Definition}
\newtheorem{notation}[thm]{Notation}
\newtheorem{conv}[thm]{Convention} \newtheorem{rem}[thm]{Remark}

\def\epsilon{\varepsilon}
\def\phi{\varphi}

\newcommand{\Curr}{\mbox{Curr}}
\newcommand{\Out}{\mbox{Out}}
\newcommand{\Aut}{\mbox{Aut}}

\newcommand{\cvn}{\mbox{cv}_N}
\newcommand{\cvnbar}{\overline{\mbox{cv}}_N}
\newcommand{\CVN}{\mbox{CV}_N}

\def\strutdepth{\dp\strutbox}
\def \ss{\strut\vadjust{\kern-\strutdepth \sss}}
\def \sss{\vtop to \strutdepth{
\baselineskip\strutdepth\vss\llap{$\diamondsuit\;\;$}\null}}

\def\strutdepth{\dp\strutbox}
\def \sst{\strut\vadjust{\kern-\strutdepth \ssss}}
\def \ssss{\vtop to \strutdepth{
\baselineskip\strutdepth\vss\llap{$\spadesuit\;\;$}\null}}

\def\strutdepth{\dp\strutbox}
\def \ssh{\strut\vadjust{\kern-\strutdepth \sssh}}
\def \sssh{\vtop to \strutdepth{
\baselineskip\strutdepth\vss\llap{$\heartsuit\;\;$}\null}}

\def\bar{\overline}

\def\strutdepth{\dp\strutbox}
\def \ss{\strut\vadjust{\kern-\strutdepth \sss}}
\def \sss{\vtop to \strutdepth{
\baselineskip\strutdepth\vss\llap{$\diamondsuit\;\;$}\null}}

\def\strutdepth{\dp\strutbox}
\def \sst{\strut\vadjust{\kern-\strutdepth \ssss}}
\def \ssss{\vtop to \strutdepth{
\baselineskip\strutdepth\vss\llap{$\spadesuit\;\;$}\null}}

\begin{document}

\begin{abstract}
We say that a subset $S\subseteq F_N$ is \emph{spectrally rigid} if whenever $T_1, T_2\in \cvn$ are points of the (unprojectivized) Outer space such that $||g||_{T_1}=||g||_{T_2}$ for every $g\in S$ then $T_1=T_2$ in $\cvn$.  It is well-known that $F_N$ itself is spectrally rigid; it also follows from the result of Smillie and Vogtmann that there does not exist a finite spectrally rigid subset of $F_N$. We prove that if $A$ is a free basis of $F_N$ (where $N\ge 2$) then almost every trajectory of a non-backtracking simple random walk on $F_N$ with respect to $A$ is a spectrally rigid subset of $F_N$.
\end{abstract}

\thanks{The author was supported by the NSF
  grant DMS-0904200}


\maketitle


\section{Introduction}

The phenomenon of ``marked length spectrum rigidity" plays an important role in the study of negatively curvature and its generalizations.
Let $M$ be a closed connected manifold and $G=\pi_1(M)$. A Riemannian metric $\rho$ on $M$ of strictly negative (but not necessarily constant) curvature defines a length function $||.||_\rho:G\to \mathbb R$, where for $g\in G$, $g\ne 1$, $||g||_\rho$ is the length of the unique closed geodesic with respect to $\rho$ representing the free homotopy class $[g]$. One also sets $||1_G||_\rho=0$.  Note that for $g\in G$ the value $||g||_\rho$ is equal to the translation length of $g$ as an isometry of the universal cover $(\widetilde M, \widetilde\rho)$.  Note also that $||g||_\rho=||g^{-1}||_\rho=||h^{-1}gh||_\rho$, so that the function $||.||_\rho$ is constant on conjugacy classes in $G$. The function $||.||_\rho$ is usually called the \emph{marked length spectrum} of $\rho$ and the image of this function (which is a subset of $\mathbb R_{\ge 0}$) is called the \emph{length spectrum} of $\rho$.
The \emph{marked length spectrum rigidity conjecture} states that
knowing the function $||.||_\rho:G\to \mathbb R$ determines the
isometry type of $(M,\rho)$. There are many variations and generalizations of this conjecture, including representation-theoretic and non-positive curvature versions.
The conjecture has been proved in a number of important cases (e.g
\cite{Cr,CEK,CFF,Ha99,HP,Otal}; see a survey of Spatzier~\cite{Sp} for
more details), but remains open in general.

For closed hyperbolic surfaces the situation is particularly nice and well-understood. Thus it is known that if $S_g$ is a closed oriented surface of genus $g\ge 2$, then there exist elements $h_1, \dots, h_{6g-5}\in G=\pi_1(S_g)$ such that whenever $\rho_1,\rho_2$ are two points in the Teichmuller space $\mathcal T(S_g)$ (i.e. marked hyperbolic metrics on $S_g$) such that $||h_i||_{\rho_1}=||h_i||_{\rho_2}$ for  $i=1,\dots , 6g-5$ then $\rho_1=\rho_2$ in  $\mathcal T(S_g)$. Thus already knowing that $||.||_{\rho_1}$ and $||.||_{\rho_2}$ agree on the finite set $\{h_1,\dots, h_{6g-5}\}$ implies that  $\rho_1=\rho_2$ in  $\mathcal T(S_g)$.  Therefore we may call the subset $\{h_1,\dots, h_{6g-5}\}\subseteq G$ ``spectrally rigid". The notion of a spectrally rigid set makes sense in any context where full marked length spectrum rigidity is already known. We think that looking for ``small'' spectrally rigid sets in such cases is an interesting general problem representing the next level in the study of length spectrum rigidity.

If $G$ is a group acting by isometries on an $\mathbb R$-tree $T$ then, under some mild assumptions (that are satisfied, in particular, in the Outer space context discussed below), it is well-known that the translation length function $||.||_T:G\to \mathbb R$ determines the action of $G$ on $T$ up to a $G$-equivariant isometry (see \cite{Pau}). However, it turns out that, unlike for the case of hyperbolic surfaces, this type of rigidity falls apart even for very nice actions if we restrict ourselves to finite subsets of $G$.
For a free group $F_N$ (where $N\ge 2$) the \emph{Culler-Vogtmann Outer space} $\cvn$ is an analog of the Teichmuller space of a hyperbolic surface. The space $\cvn$ consists of minimal free discrete isometric actions of $F_N$ on $\mathbb R$-trees, considered up to $F_N$-equivariant isometry. Every element $T\in \cvn$ arises as the universal cover of a finite graph $\Gamma$, whose fundamental group is identified with $F_N$ via a particular isomorphism, where edges of $\Gamma$ are given positive real lengths and their lifts to $T$ are given the same lengths. There is an important subset, $\CVN\subseteq \cvn$, consisting of all $T\in \cvn$ such that the quotient metric graph $T/F_N$ has volume 1. The space $\CVN$ is the  \emph{projectivized Culler-Vogtmann Outer space}, which in fact was introduced first, in \cite{CV}. Both $\cvn$ and $\CVN$ play an important role in the study of $Out(F_N)$.
We say that a subset $R\subseteq F_N$ is \emph{spectrally rigid} if whenever $T_1, T_2\in \cvn$ are such that $||g||_{T_1}=||g||_{T_2}$ for every $g\in R$ then $T_1=T_2$ in $\cvn$. As noted above, $R=F_N$ is spectrally rigid and in fact it is enough to make $R$ consist of representatives of all conjugacy classes in $F_N$.
A surprising result of Smillie and Vogtmann~\cite{SV} shows that there
does not exist a finite spectrally rigid subsets of $F_N$, where $N\ge
3$. They prove, in particular,  that for any finite subset $R\subseteq
F_N$ there exists a one-parametric family $(T_t)_{t\in [0,1]}$ of
distinct points of $\CVN$ such that for every $t\in [0,1]$ the length
functions $||.||_{T_0}$ and $||.||_{T_t}$ agree on $R$.  A similar
statement follows from the recent work of  Duchin, Leininger, and Rafi
for flat metrics on surfaces of finite type~\cite{DLR}. Moreover, the paper \cite{DLR} gives a complete characterization of when a set of simple closed curves on a finite type surface is spectrally rigid with respect to the space of flat metrics on that surface.

As noted above, we believe that for any context where full length spectrum rigidity is known it becomes interesting to look for ``small" (in some reasonable sense) spectrally rigid subsets of the group $G$ in question and to try to characterize which subsets of $G$ are spectrally rigid. Apart from the above mentioned paper of Duchin, Leininger, and Rafi, we are not aware of any other results in this direction. In the present paper, as well as in our joint work with Carette~\cite{CK}, we study infinite but "sparse" subsets of $F_N$ that are spectrally rigid with respect to the collection of length functions on $F_N$ coming from $\cvn$.

Let $N\ge 2$ and let $A$ be a free basis of $F_N$. Consider a simple non-backtracking random walk on $F_N$ corresponding to $A$, starting at $1\in F_N$. Thus a trajectory of this walk is a semi-infinite freely reduced word $\xi=x_1x_2\dots x_n\dots $ where $x_i\in A^{\pm 1}$. For $n\ge 1$ denote $\xi(n)=x_1x_2\dots x_n\in F_N$, the initial segment of $\xi$ of length $n$.
 
 We prove the following:
 
 \begin{thm}\label{main}
 Let $N\ge 2$ and let $A$ be a free basis of $F_N$.
 For a.e. trajectory $\xi$ of the non-backtracking simple random walk on $F_N$ with respect to $A$ the following holds:
 
 For every $n_0\ge 1$ the subset
 \[
 \{\xi(n): n\ge n_0\}\subseteq F_N
 \]
 is spectrally rigid.
 \end{thm}
 
 It is known (c.f. \cite{KKS}) that for a.e. trajectory $\xi$ of the walk above we have
 \[
 \lim_{n\to\infty} \frac{||\xi(n)||_A}{n}=1.
 \]
 Here for $g\in F_N$,  $||g||_A$ is the cyclically reduced length of $g$ over $A$. Thus $||\xi(n)||_A=n+o(n)$, which shows that the set $\{\xi(n): n\ge n_0\}$ in Theorem~\ref{main} represents a ``sparse" collection of conjugacy classes in $F_N$.

The proof of Theorem~\ref{main} relies on the machinery of \emph{geodesic currents} on free groups (see Section~\ref{sect:background} for the definition and basic properties) and particularly the notion of a ``geometric intersection form'' between trees and currents, developed in the work of the author and Martin Lustig~\cite{Ka2, KL2,KL3}. The presence of $\mathbb R_{\ge 0}$-linear structure on the space of currents turns out to be particularly useful.
Let us point out the new ideas introduced in this paper. In the earlier work~\cite{Ka2,KL2} we have constructed a "length pairing" or a "geometric intersection form" 
\[
\langle \cdot\, , \ \cdot \rangle: \cvn\times Curr(F_N)\to \mathbb R_{\ge 0}
\] 
where $Curr(F_N)$ is the space of geodesic currents on $F_N$ (see Section~\ref{sect:background} below for details). This geometric intersection form turns out to have a number of nice properties, one of which is being $\mathbb R_{\ge 0}$-linear in the second argument. The elements of $Curr(F_N)$, that is geodesic currents, are traditionally defined as \emph{positive} Borel measures on $\partial^2F_N=\partial F_N\times \partial F_N - diag$ satisfying some additional properties, namely being $F_N$-invariant, flip-invariant and finite on compact sets. It turns out that for the proof of Theorem~\ref{main} we need to extend the geometric intersection form to the space of \emph{signed} geodesic currents, which satisfy all the axioms of the standard currents, except that as measures they are no longer required to be positive and are allowed to take arbitrary (but finite) real values on compact subsets.  The extended intersection form now takes values in $\mathbb R$ and enjoys $\mathbb R$-linearity in the second argument, which is the space $\mathcal SCurr(F_N)$ of all signed currents on $F_N$. Crucially, the $\mathbb R$-linear span of the standard rational currents is dense in $\mathcal SCurr(F_N)$ (see Proposition~\ref{prop:dense} below). 
It is well-known, and fairly easy to see~\cite{KKS} that for a random trajectory $\xi$ as in Theorem~\ref{main} we have
\[
\lim_{n\to\infty}\frac{\eta_{\xi(n)}}{n}=\nu_A
\]
where $\eta_{\xi(n)}$ is the rational current corresponding to $\xi(n)$ and where $\nu_A$ is the \emph{unform current} on $F_N$. Thus, projectively, the set  $\{\xi(n): n\ge n_0\}$ (viewed as a subset of $Curr(F_N)$) in Theorem~\ref{main} is rather small: after rescaling, this set represents a sequence converging to a single point of $Curr(F_N)$. A key step in the proof of Theorem~\ref{main} is Proposition~\ref{prop:main} which says that the $\mathbb R$-linear span of the set $\{\eta_{\xi(n)}: n\ge n_0\}$ is dense in $\mathcal SCurr(F_N)$. Informally,  Proposition~\ref{prop:main} says that, although the sequence $\frac{\eta_{\xi(n)}}{n}$ converges to a single point in $\mathcal SCurr(F_N)$ (namely $\nu_A)$, it converges to that point "from all possible directions". Once Proposition~\ref{prop:main}  is established, it is not hard to derive Theorem~\ref{main} using the basic properties of the geometric intersection form.

\begin{rem}\label{rem:orbit}
In an ongoing joint work with Mathieu Carette~\cite{CK} we prove that the set of primitive elements in a free group $F_N$ (where $N\ge 2$) is spectrally rigid, and that, moreover, for $N\ge 3$ every $Aut(F_N)$-orbit of a nontrivial element of $F_N$ is spectrally rigid as well. The proofs there require substantially different methods, as it is not hard to show (by using Whitehead graph properties of primitive elements)  that, unlike for the spectrally rigid sets constructed in the present paper, the linear span of the set of rational currents of primitive elements is not dense in the space of signed geodesic currents.  The result of \cite{CK} raises a natural question of when for a subgroup $H\le Aut(F_N)$ and an element $f\in F_n$ the orbit $Aut(F_N)f$ is spectrally rigid.

It turns out that for the case where $H$ is cyclic generated by an atoroidal fully irreducible automorphism, an $H$-orbit is never spectrally rigid.
As noted above, Smillie and Vogtmann proved~\cite{SV}
that for any finite set $R\subseteq F_N$ (where $N\ge 3$), the set $R$
is not spectrally rigid. In fact, the proof of Theorem~1 in \cite{SV}
goes through verbatim and shows that any (possibly infinite) subset
$R\subseteq F_N$ (again where $N\ge 3$) is not spectrally rigid provided it satisfies the
following ''weak aperiodicity'' property: there is a free basis $A$ of $F_N$, an element
$a\in A$ and an integer $M\ge 1$ such that whenever $a^k$ is a subword
of a cyclically reduced word over $A$ representing an element
conjugate to some $g\in R$ then $|k|\le M$. Let $N\ge 3$ and let
$\phi\in \Aut(F_N)$ be atoroidal and fully irreducible (also known as "iwip", which stands for irreducible with irreducible powers). Proposition~1.18 in \cite{BFH97}, together with
basic train-track properties, then implies
that for any nontrivial $g\in F_N$ the set $\{\phi^n(g): n\in \mathbb
Z\}$ satisfies the ''weak aperiodicity'' property. Hence the set
$\{\phi^n(g): n\in \mathbb Z\}$ is not spectrally rigid in $F_N$.
\end{rem}

The author is grateful to Chris Leininger and Lewis Bowen for helpful discussions.

\section{Outer space and the space of geodesic currents}\label{sect:background}

We give here only a brief overview of basic facts related to Outer space and the space of geodesic currents. We refer the reader to~\cite{Martin,Ka1,Ka2} for more detailed background information about geodesic currents. Other useful references regarding currents and their applications are \cite{KL1,KN,KL2, KL3,Fra,BF08,Ha,CP,KL5,CHL3}. 

\subsection{Outer space} Let $N\ge 2$. The \emph{unprojectivized Outer space} $\cvn$ consists of all minimal free and discrete isometric actions on $F_N$ on $\mathbb R$-trees (where two such actions are considered equal if there exists an $F_N$-equivariant isometry between the corresponding trees). There are several different topologies on $\cvn$ that are known to coincide, in particular the equivariant Gromov-Hausdorff convergence topology and the so-called \emph{axis} or \emph{length function} topology. Every $T\in \cvn$ is uniquely determined by its \emph{translation length function} $||.||_T:F_N\to\mathbb R$, where $||g||_T$ is the translation length of $g$ on $T$. Two trees $T_1,T_2\in\cvn$ are close if the functions $||.||_{T_1}$ and $||.||_{T_1}$ are close pointwise on a large ball in $F_N$. The closure $\cvnbar$ of $\cvn$ in either of these two topologies is well-understood and known to consists precisely of all the so-called \emph{very small} minimal isometric actions of $F_N$ on $\mathbb R$-trees, see \cite{BF93} and \cite{CL}.

A \emph{simplicial chart} or \emph{marking} on $F_N$ is an isomorphism $\alpha: F_N\to \pi_1(\Gamma)$ where $\Gamma$ is a finite connected graph without degree-one and degree-two vertices. We think of $\Gamma$ as an oriented graph, where every edge $e$ comes equipped with an inverse edge $e^{-1}\ne e$ such that $(e^{-1})^{-1}=e$, and such that the initial vertex of $e$ is the terminal vertex of $\bar e$ and vise versa. An \emph{orientation} on $\Gamma$ is a partition $E\Gamma=E^+\Gamma\sqcup E^-\Gamma$ such that for every $e\in E\Gamma$ exactly one of $e,\bar e$ belongs to $E^+\Gamma$.
 A \emph{metric graph structure} on $\Gamma$ is a function $\mathcal L: E\Gamma\to \mathbb R_{>0}$ such that $\mathcal L(e)=\mathcal L(e^{-1})$ for every $e\in E\Gamma$.  A \emph{marked metric graph} structure on $F_N$ is a pair $(\alpha, \mathcal L)$  where  $\alpha: F_N\to \pi_1(\Gamma)$ is a marking and where $\mathcal L$ is a metric graph structure on $\Gamma$. Every marked metric graph structure $(\alpha, \mathcal L)$  determines a point of $\cvn$ obtained by taking the universal covering (which is a topological tree) $\ widetilde \Gamma$ of $\Gamma$ and lifting the metric structure $\mathcal L$ to $\widetilde \Gamma$. Moreover, every point of $\cvn$ arises in this way. For a fixed chart $\alpha$ varying the metric structure $\mathcal L$ on $\Gamma$ determines an open cone inside $\cvn$.

\subsection{Geodesic currents} Let $\partial^2F_N:=\{ (x,y)| x,y\in \partial F_N, x\ne y\}$. The action of $F_N$ by translations on its hyperbolic boundary $\partial F_N$ defines a natural diagonal action of $F_N$ on $\partial^2 F_N$. A \emph{geodesic current} on $F_N$ is a positive Borel measure on $\partial^2 F_N$, finite on compact subsets,  which is $F_N$-invariant and is also invariant under the ``flip" map $\partial^2 F_N\to \partial^2 F_N$, $(x,y)\mapsto (y,x)$. The space $\Curr(F_N)$ of all geodesic currents on $F_N$ has a natural $\mathbb R_{\ge 0}$-linear structure and is equipped with the weak*-topology
of pointwise convergence on continuous functions. Every point $T\in \cvn$ defines a \emph{simplicial chart} on $\Curr(F_N)$ which allows one to think about geodesic currents as systems of nonnegative weights satisfying certain Kirchhoff-type equations; see \cite{Ka2} for details. 
We briefly recall this description here. Let $\alpha:F_N\to \pi_1(\Gamma)$ be a simplicial chart on $F_N$. We view $\widetilde \Gamma$ as a simplicial tree, and thus an element of $\cvn$, by giving every edge length 1. Then $F_N$ is $F_N$-equivariantly quasi-isometric to $\widetilde \Gamma$, which determines a homeomorphism $\partial F_N\to \partial \widetilde\Gamma$, which we will view as an identification between them. For a nondegenerate geodesic segment $\gamma=[p,q]$ in $\widetilde\Gamma$ the \emph{two-sided cylinder} $Cyl_\Gamma(\gamma)\subseteq \partial^2 F_N$ consists of all $(x,y)\in \partial^2 F_N$ such that the geodesic from $x$ to $y$ in $\widetilde\Gamma$ passes through $\gamma=[p,q]$.  Given a nontrivial reduced (i.e. without backtracking) edge-path $v$ in $\Gamma$ and a current $\mu\in \Curr(F_N)$, the ``weight" $\langle v,\mu\rangle_\alpha$ is defined as $\mu(Cyl_\Gamma(\gamma))$ where $\gamma$ is any lift of $v$ to  $\widetilde \Gamma$ (the fact that the measure $\mu$ is $F_N$-invariant implies that a particular choice of a lift of $v$ does not matter). A current $\mu$ is uniquely determined by the family of weights $\langle v,\mu\rangle_\alpha$ where $v$ varies over all nontrivial reduced finite edge-paths in $\Gamma$. The weak*-topology
on $\Curr(F_N)$ corresponds to pointwise convergence of the weights for every such $v$. 

If $A$ is a free basis of $F_N$, then the wedge of circles corresponding to the elements of $A$ determines a simplicial chart on $F_N$; the universal cover of this wedge of circles is the Cayley graph $T_A$ of $F_N$ with respect to $A$. In this case the weights are indexed by nontrivial freely reduced words $v$ over $A$ and we denote the corresponding weights by $\langle v,\mu\rangle_A$.


For every $g\in F_N, g\ne 1$ there is an associated \emph{counting current} $\eta_g\in \Curr(F_N)$. If $\alpha:F_N\to\pi_1(\Gamma)$ is a simplicial chart  on $F_N$, the conjugacy class $[g]$ of $g$ can be realized realized by a reduced ``cyclic path" $w$, that is a graph immersion from a clock-wise oriented simplicial circle to $\Gamma$ that represents $[\alpha(g)]$. We can think of edges of $w$ as being labelled by edges of $\Gamma$. Then for every nontrivial freely reduced edge-path $v$ in $\Gamma$  the weight $\langle v,\eta_g\rangle_\alpha$ is equal to the total number of occurrences of $v^{\pm 1}$ in $w$, where an \emph{occurrence} of $v$ in $w$ is a vertex on $w$ such that we can read $v$ in $w$ clockwise without going off the circle. We refer the reader to \cite{Ka2} for a detailed exposition on the topic. By construction the counting current $\eta_g$ depends only on the conjugacy class $[g]$ of $[g]$ and it also satisfies $\eta_g=\eta_{g^{-1}}$. One can check~\cite{Ka2} that for $\phi\in \Out(F_N)$ and $g\in F_N, g\ne 1$ we have $\phi\eta_g=\eta_{\phi(g)}$. Scalar multiples $c\eta_g\in \Curr(F_N)$, where $c\ge 0$, $g\in F_N, g\ne 1$ are called \emph{rational currents}. A key fact about $\Curr(F_N)$ states that the set of all rational currents is dense in $\Curr(F_N)$\cite{Ka1,Ka2}.  

If $A$ is a free basis of $F_N$ and $\Gamma_A$ is the wedge of circles corresponding to elements of $A$, a reduced cyclic path $w$  in $\Gamma_A$ can be viewed as a reduced \emph{cyclic word} over $A$, that is, a cyclically reduced word over $A$ written on a simplicial circle in the clockwise direction, without specifying the base-point of the circle. 

The spaces $\cvn$ and $\Curr(F_N)$ come equipped with natural actions of $\Out(F_N)$. However, the details of how these actions are defined are not relevant for this paper and we omit them.

\subsection{Intersection form}

In \cite{KL2} Kapovich and Lustig  constructed a natural geometric \emph{intersection form} which pairs trees and currents:

\begin{prop}\label{int}\cite{KL2}
Let $N\ge 2$. There exists a unique continuous map $\langle\ ,\ \rangle : \cvnbar \times \Curr(F_N)\to \mathbb R_{\ge 0}$ with the following properties:
\begin{enumerate}
\item We have $\langle T, c_1\mu_1+c_2\mu_2\rangle=c_1\langle T,\mu_1\rangle+c_2\langle T,\mu_2\rangle$ for any $T\in \cvnbar$, $\mu_1,\mu_2\in \Curr(F_N)$, $c_1,c_2\ge 0$.
\item We have $\langle cT, \mu\rangle=c\langle T,\mu\rangle$ for any $T\in\cvnbar$, $\mu\in \Curr(F_N)$ and $c\ge 0$.
\item We have $\langle T\phi,\mu\rangle=\langle T, \phi\mu\rangle$ for any $T\in\cvnbar$, $\mu\in \Curr(F_N)$ and $\phi\in \Out(F_N)$.
\item We have $\langle T, \eta_g\rangle=||g||_T$ for any $T\in \cvnbar$ and $g\in F_N, g\ne 1$.
\item  For every $\mu\in \Curr(F_N)$ and for every tree $T\in \cvn$ given by a simplicial chart $\alpha:F_N\to \pi_1(\Gamma)$ with a metric structure $\mathcal L$ on $\Gamma$ we have
\[
\langle T, \mu\rangle=\sum_{e\in E^+\Gamma} \langle e, \mu\rangle_\alpha \mathcal L(e).
\]
\end{enumerate}
\end{prop}

Recall that  $\cvnbar$ is the closure of $\cvn$ in the length function topology, or, equivalently, in the equivariant Gromov-Hausdorff convergence topology.  The case of $\cvn \times \Curr(F_N)$ was dealt with in an earlier papers~\cite{Lu1,Ka2}. 

\section{Uniform measure on the boundary}

\begin{conv}
From now and until the end of the paper, unless specified otherwise, we fix a free basis $A=\{a_1,\dots, a_N\}$ of $F_N$, where $N\ge 2$.
\end{conv}

Recall that $T_A$ is the Cayley graph of $F_N$ with respect to $A$. 
For a word $v$ (not necessarily freely reduced) we denote by $|v|$ the length of $v$, denote by $|v|_A$ the freely reduced length of $v$ and denote by $||v||_A$ the cyclically reduced length of $v$. Thus $||v||_A=||v||_{T_A}$.  We will think of elements of $F_N$ as freely reduced words over $A$, so that when writing $v\in F_N$, we will implicitly assume that $v$ is a freely reduced word.
 
The boundary $\partial F_N$ consists of geodesic rays originating at the vertex $1$ in $T_A$. Every such ray $\xi$ can also be viewed as a semi-infinite freely reduced word $\xi=x_1x_2\dots x_n\dots $ over $A$, where $x_i\in A^{\pm 1}$.  For $n\ge 1$ we denote $\xi(n)=x_1\dots x_n\in F_N$, the initial segment of $\xi$ of length $n$.  If $v$ is a nontrivial freely reduced word over $A$, denote by $Cyl_A(v)$ the set of all $\xi\in \partial F_N$ such that $v$ is an initial segment of $\xi$. The set $Cyl_A(v)\subseteq \partial F_N$ is the \emph{one-sided cylinder} defined by $v$. The set of all one-sided cylinders forms a basis of open (also closed) sets for the topology on $\partial F_N$.  Any finite Borel measure on $\partial F_N$ is determined by its values on all the cylinders $Cyl_A(v)$, $v\in F_N, v\ne 1$.

For $M\ge 1$ denote by $S_A(M)$ the set of all freely reduced words over $A$ in $F_N$ of length $M$,  so that $S_A(M)$ is the sphere of radius $M$ in the Cayley graph of $F_N$ corresponding to the generating set $A$. Then $d(M):=\#(S_A(M))=2N(2N-1)^{M-1}$.

\begin{defn}[Uniform measure]
The \emph{uniform measure} $m_A$ corresponding to $A$ is a Borel probability measure on $\partial F_N$ given by the condition
\[
m_A(Cyl_A(v))= \frac{1}{d(|v|_A)}, \quad v\in F_N, v\ne 1.
\]
\end{defn}
It is not hard to see that the above condition indeed defines a Borel probability measure on $\partial F_N$. Moreover, this measure is exactly the exit measure of the simple non-backtracking random walk on $F_N=F(A)$ starting at $1\in F_N$. We refer the reader to \cite{KKS} for more details. Note also that for $m_A$-a.e. $\xi\in \partial F$ the semi-infinite freely reduced word $\xi=x_1x_2\dots x_n\dots $ contains every finite freely reduced word $v\in F(A)$ as a subword infinitely many times. This fact is a consequence of the law of large numbers and one can say something more precise: if $\langle v, \xi(n)\rangle$ is the number of times that $v$ occurs as a subword of $\xi(n)=x_1\dots x_n$ then
\[
\lim_{n\to\infty} \frac{langle v, \xi(n)\rangle}{n}=\frac{1}{\#S_A(M)}.
\]
This observation is responsible for the fact (see Remark~\ref{uc} below) that for $m_A$-a.e. $\xi$ we have 
\[
\lim_{n\to\infty}\frac{\eta_{\xi(n)}}{n}
\]
where $\nu_A$ is the \emph{uniform current} on $F_N$ corresponding to $A$.

\section{Finite-dimensional approximations of $\Curr(F_N)$}

\subsection{Inverse limit description of $\Curr(F_N)$}
We briefly recall here an inverse limit description of  $\Curr(F_N)$ by finite-dimensional approximations, see \cite{Ka1,Ka2} for more details.

Let $M\ge 1$.  Consider the euclidean space $\mathbb R^{d(M)}$ with coordinates indexed by elements $v\in S_A(M)$. We write an element $p\in \mathbb R^{d(M)}$ as $p=(p_v)_{v\in S_A(M)}$ where $p_v$ is the $v$-th coordinate of $p$.
Now let $Z_M\subseteq \mathbb R^{d(M)}$ be the set of all $p\in \mathbb R^{d(M)}$ such that t$p$ satisfies the following conditions:
\begin{gather*}
p_v\ge 0 \text{ for every } v\in S_A(M)\\
p_v=p_{v^{-1}}  \text{ for every } v\in S_A(M)\\
\sum_{a\in A^{\pm 1}: |ua|_A=M} p_{ua}\quad =\sum_{b\in A^{\pm 1}: |bu|_A=M} p_{bu} \quad \text{ for every   } u\in S_A(M-1)
\end{gather*}
If $M=1$, the last condition above is absent when $Z_1$ is defined.

Thus $Z_M\subseteq \mathbb R^{d(M)}$ is a closed polyhedral cone. 
For $M\ge 2$ define the projection map $\pi_M: Z_M\to Z_{M-1}$ by the formula 
\[
(\pi_M(p))_u=\sum_{a\in A^{\pm 1}: |ua|_A=M} p_{ua}
\]
where $p\in Z_M$ and $u\in S_A(M-1)$. Then (see \cite{Ka1,Ka2}) the map $\pi_M: Z_M\to Z_{M-1}$ is onto.
Moreover, the inverse limit of the sequence of maps \[ \overset{\pi_{M+1}}{\to} Z_M  \overset{\pi_M}{\to} Z_{M-1} \overset{\pi_{M-1}}{\to} \dots  \overset{\pi_2}{\to} Z_1  \tag{\dag}\] is naturally homeomorphic to $\Curr(F_N)$. It is useful to see this identification explicitly. For a current $\mu\in \Curr(F_N)$ and $M\ge 1$ let $\tau_M(\mu)=p\in Z_M$ where $p_v=\langle v, \mu\rangle_A$ for every $v\in S_A(M)$. Then $\pi_M\circ \tau_M=\tau_{M-1}$ for every $M\ge 2$. The sequence $(\tau_M(\mu))_{M\ge 1}$ is an element of the inverse limit of $(\dag)$  corresponding to $\mu\in \Curr(F_N)$.  Note that the maps $\pi_M$ and $\tau_M$ and $\mathbb R_{\ge 0}$-linear.

For a cyclic word $w$ over $A$ denote $\tau_M(w):=\tau_M(\eta_w)\in Z_M$.

The following lemma is an immediate consequence of the definitions and of the fact that for a cyclic word $w$ the weight $\langle v, \eta_w\rangle_A$ is equal to the number of occurrences of $v^{\pm 1}$ in $w$. 

\begin{lem}\label{lem:key}
Let $M\ge 2$ and let $u\in S_A(M)$. Then for any freely reduced word $w$ over $A$ of length $\ge 2M$ such that the words $w$ and $wu$ are cyclically reduced as written, the difference $\tau_M(wu)-\tau_M(w)$ depends only on $u$ and the initial and terminal segments of $w$ of length $M$.
\end{lem}

\begin{cor}\label{cor:key}
Let $M\ge 2$ and let $u\in S_A(M)$ be cyclically reduced.  Suppose $w$ is a freely reduced word of length $\ge 2M$ such that $u$ and $w$ have the same initial segment of length $M$ and such that $u$ and $w$ have the same terminal segment of length $M$. Then:

\begin{enumerate}
\item We have  $\tau_M(wu)-\tau_M(w)=\tau_M(u)$.
\item For every $2\le M' \le M$ we have  $\tau_{M'}(wu)-\tau_{M'}(w)=\tau_{M'}(u)$.
\end{enumerate}
\end{cor}

\begin{proof}
By Lemma~\ref{lem:key} we may assume that $w=u^2$. Then $\tau_M(wu)=\tau_M(u^3)=3\tau_M(u)$, $\tau_M(w)=\tau_M(u^2)=2\tau_M(u)$ and hence $\tau_M(wu)-\tau_M(w)=\tau_M(u)$. This establishes part (1) of the corollary. Part (2) immediately follows from part (1).
\end{proof}

\subsection{Signed currents and extended intersection form.}

For technical reasons in these paper we need to work with currents that are not necessarily positive as measures.
Let $\mathcal SCurr(F_N)$ be the space of all $F_N$-invariant and flip invariant locally finite (finite on compact sets) Borel measures on $\partial_2 F_N$ that are not necessarily positive.  We call elements of $\mathcal SCurr(F_N)$ \emph{signed geodesic currents}. Thus $\mathcal SCurr(F_N)$ is a vector space over $\mathbb R$.  As before, we equip $\mathcal SCurr(F_N)$ with the weak-* topology.  For a simplicial chart $\alpha:F_N\to \pi_1(\Gamma)$ and a measure $\mu\in \mathcal SCurr(F_N)$ the weights $\langle v,\mu\rangle_\alpha$ are defined in the same way as for elements of $\Curr(F_N)$. Also, similar to the $\Curr(F_N)$ situation, we have that $\lim_{n\to\infty} \mu_n =\mu$ , where $\mu_n,\mu\in \mathcal SCurr(F_N)$ if and only if for every finite non-degenerate reduced path $v$ in $\Gamma$ we have $\lim_{n\to\infty} \langle v, \mu_n\rangle_\alpha=\langle v,\mu\rangle_\alpha$. The set $\Curr(F_N)$ is a closed $\mathbb R_{\ge 0}$-linear subspace of $\mathcal SCurr(F_N)$.
We can extend the intersection form $\langle \ , \ \rangle: \cvn\times \Curr(F_N)\to \mathbb R_{\ge 0}$ to $\langle \ ,\  \rangle: \cvn\times \mathcal SCurr(F_N)\to \mathbb R$ by the formula
\[
\langle T, \mu\rangle=\sum_{e\in E^+\Gamma} \langle e, \mu\rangle_\alpha \mathcal L(e).
\]
where $\mu\in \mathcal SCurr(F_N)$ and $T\in \cvn$ is given by a simplicial chart $\alpha:F_N\to \pi_1(\Gamma)$ with a metric structure $\mathcal L$ on $\Gamma$. Then exactly the same argument as in \cite{Ka2} yields:  
\begin{prop}\label{prop:ext}
The map $\langle \ ,\  \rangle: \cvn\times \mathcal SCurr(F_N)\to \mathbb R$ is continuous, $\mathbb R_{\ge 0}$-homogeneous in the first argument, $\mathbb  R$-linear in the second argument and $Out(F_N)$-invariant (with respect to the natural left actions of $Out(F_N)$ on $\cvn$ and $\mathcal SCurr(F_N)$).
\end{prop}
It is obvious from the definition and from Proposition~\ref{int} that this extended intersection form agrees with the old one on $\cvn\times \Curr(F_N)$.
Moreover, we also get a sequence of finite-dimensional approximations of $\mathcal SCurr(F_N)$ similar to that of $\Curr(F_N)$. Namely, we define $\widetilde Z_M$ exactly as $Z_M$ but omitting the condition $p_v\ge 0, v\in S_A(M)$. Then $\widetilde Z_M\subseteq \mathbb R^{d(M)}$ is a $\mathbb R$-linear subspace containing $Z_M$, the map $\pi_M:Z_M\to Z_{M-1}$ naturally extends to a $\mathbb R$-linear map $\widetilde\pi_M:\widetilde Z_M\to\widetilde Z_{M-1}$ (given by the same formula as $\pi_M$) and the map $\tau_M:\Curr(F_N)\to Z_M$ naturally extends to a $\mathbb R$-linear continuous map  $\widetilde\tau_M:\mathcal SCurr(F_N)\to \widetilde Z_M$ given by the same formula as $\tau_M$.  Also, we still have the property $\widetilde \pi_M\circ \widetilde\tau_M=\widetilde\tau_{M-1}$ for every $M\ge 2$. Note that knowing all the weights $\langle v,\mu\rangle_A$, where $v\in S_A(M)$, uniquely determines all the "lower-level" weights  $\langle u,\mu\rangle_A$, where $1\le |u|_A\le M$, by iterating linear formulas corresponding to the last equation in the definition of $Z_M$.  

One can show that $\mathcal SCurr(F_N)$ is equal to the inverse limit of the sequence $\widetilde \pi_M$, although this fact is not needed in this paper.

\begin{notation}
If $X$ is a $\mathbb R$-vector space and $Y\subseteq X$, we denote by
$Span(Y)$ the linear span of $Y$ in $X$, that is, the set of all
finite $\mathbb R$-linear combinations of elements of $Y$.
\end{notation}

\begin{prop}\label{prop:dense}
Let $N\ge 2$. Then the following hold:
\begin{enumerate}
\item For every $\mu\in \mathcal SCurr(F_N)$ there exist $\mu_1,\mu_2\in
  Curr(F_N)$ such that $\mu=\mu_1-\mu_2$.
\item The span $Span\big( \{\eta_{g}: g\in F_N-\{1\}\} \big)$ is dense
  in $\mathcal SCurr(F_N)$.
\end{enumerate}
\end{prop}
\begin{proof}
Note that part (1) directly implies part (2) since it is known that
the set of rational currents is dense in $Curr(F_N)$.

To see that part (1) holds, recall that by the Hahn-Jordan
Decomposition Theorem for
any $\mu\in \mathcal SCurr(F_N)$ there exists a decomposition $\partial^2
F_N= E_1\sqcup E_2$ where $E_1, E_2$ are Borel sets and where for
every Borel subset $S\subseteq E_1$ we have $\mu(S)\ge 0$ and for every
Borel subset $S\subseteq E_2$ we have $\mu(S)\le 0$. 
Moreover, in view of $F_N$-invariance and flip-invariance of $\mu$, the
uniqueness part of of the Hahn-Jordan Theorem easily implies that the
sets $E_1, E_2$ may be assumed to be $F_N$-invariant and
flip-invariant.

Define measures $\mu_1,\mu_2$
on $\partial^2 F_N$ by $\mu_1(S):=\mu(S\cap E_1)$ and
$\mu_2(S):=-\mu(S\cap E_2)$, where $S\subseteq \partial^2 F_N$. Then
$\mu=\mu_1-\mu_2$ and $\mu_1,\mu_2$ are positive Borel measures on
$\partial^2F$ that are finite on compact subsets. Moreover, since $E_1, E_2$ are $F_N$-invariant and
flip-invariant, the measures $\mu_1,\mu_2$ are also $F_N$-invariant and
flip-invariant, so that $\mu_1,\mu_2\in Curr(F_N)$.
\end{proof}

\section{Linear spans and random rays}

\begin{prop}\label{prop:main}
There exists a subset of $\partial F_N$ of $m_A$-measure 1 such that every point $\xi$ from that subset has the following property:

For every $n_0\ge 1$ we have

\[
\mathcal SCurr(F_N)=\overline{Span\big( \{\eta_{\xi(n)}: n\ge n_0\} \big)}
\]
\end{prop}

\begin{proof}

Let $\Omega\subseteq \partial F_N$ be the set of all $\xi\in \partial F_N$ such that every nontrivial freely reduced word over $A$ occurs as a subword of $\xi$ infinitely many times. Then, as is well-known, $m_A(\Omega)=1$.

Let $\xi\in \Omega$ be arbitrary. 

Let $n_0\ge 2$ and let $w$  be an arbitrary nontrivial cyclically reduced word over $A$.
Let $M\ge 2$ be an arbitrary integer. Let $v$ be the initial segment of $\xi$ of length $M$. We can always choose freely reduced words  $s,t$ of length $M$ such that the word $vs(w^nt)^m$ is freely and cyclically reduced as written for all $m,n\ge 1$.  

Let $n\ge 1$ be arbitrary. Since $\xi\in \Omega$, the word $vs(w^nt)^2$ occurs as a subword of $\xi$ infinitely many times. Choose such an occurrence that starts at position $\ge n_0+1$ in $\xi$. Thus there exists an initial segment of $\xi$ of the form $vf_nvs(w^nt)^2$ where $|vf_n|\ge n_0+1$. Then by Corollary~\ref{cor:key} we have $\tau_M(vf_nvs(w^nt)^2)-\tau_M(vf_nvsw^nt)=\tau_M(vsw^nt)$.  Hence
\[
\lim_{n\to\infty} \frac{1}{n}\left( \tau_M(vf_nvs(w^nt)^2)-\tau_M(vf_nvsw^nt)  \right)=\tau_M(w).
\]

Then (in view of the last equation in the definition of $Z_M$ and $\widetilde Z_M)$ we can find $n=n_M\ge 1$ such that for every freely reduced word $u$ over $A$ with $0<|u|\le M$ we have \[\left|\langle u, \frac{1}{n}(\eta_{vf_nvs(w^nt)^2}-\eta_{vf_nvsw^nt})\rangle-\langle u, \eta_w\rangle\right|\le 2^{-M}\]
Put $\mu_M=\frac{1}{n}(\eta_{vf_{n}vs(w^{n}t)^2}-\eta_{vf_{n}vsw^{n}t})$
Therefore, by definition of topology on $\mathcal SCurr(F_N)$, we have
$\lim_{M\to\infty} \mu_M=\eta_w$ in $\mathcal SCurr(F_N)$. Hence $\eta_w\in
\overline{Span\big( \{\eta_{\xi(n)}: n\ge n_0\} \big)}$. 
Proposition~\ref{prop:dense} now implies that
\[
\mathcal SCurr(F_N)=\overline{Span\big( \{\eta_{\xi(n)}: n\ge n_0\} \big)},
\]
as required.

\end{proof}

We can now prove the main result, Theorem~\ref{main} from the Introduction:
\begin{thm}\label{thm:main}
There exists a subset of $\partial F_N$ of $m_A$-measure 1 such that every point $\xi$ from that subset has the following property:

For every $n_0\ge 1$ the set

\[
 \{\xi(n): n\ge n_0\}\subseteq F_N
\]
is spectrally rigid.
\end{thm}
\begin{proof}
Let $\Omega\subseteq \partial F_N$ be the set with $m_A(\Omega)=1$ provided by Proposition~\ref{prop:main}.
Let $\xi\in \Omega$ be arbitrary and let $n_0\ge 1$. Suppose $T_1,T_2\in \cvn$ are such that for every $n\ge n_0$ we have $||\xi(n)||_{T_1}=||\xi(n)||_{T_2}$.
Hence for every $n\ge n_0$
\[
\langle T_1, \eta_{\xi(n)}\rangle=\langle T_2, \eta_{\xi(n)}\rangle.
\] 
By linearity and continuity of the intersection form on  $\cvn\times \mathcal SCurr(F_N)$ it follows that
\[
\langle T_1, \mu\rangle=\langle T_2,\mu\rangle
\]
for every $\mu\in \overline{Span\big( \{\eta_{\xi(n)}: n\ge n_0\} \big)}$. Hence, by Proposition~\ref{prop:main}, for every nontrivial $w\in F_N$ we have
\[
||w||_{T_1}=\langle T_1, \eta_w\rangle=\langle T_2, \eta_w\rangle=||w||_{T_2}.\] 
Thus the length functions $||.||_{T_1}$ and $||.||_{T_2}$ are equal on $F_N$ and hence $T_1=T_2$ in $\cvn$. This shows that the set  $\{\xi(n): n\ge n_0\}$
is spectrally rigid in $F_N$, as claimed.
\end{proof}

\begin{rem}\label{uc}
As noted in \cite{KKS}, for an $m_A$-a.e. point $\xi\in \partial F_N$ we have $\lim_{n\to\infty} \frac{1}{n}\eta_{\xi(n)}=\nu_A$. Here $\nu_A\in \Curr(F_N)$ is the \emph{uniform current} corresponding to $A$, which is given by the weights $\langle v, \nu_A\rangle=\frac{1}{d(|v|_A)}$ for every nontrivial freely reduced word $v$ over $A$. The proof of Theorem~\ref{thm:main} is based on exploiting the fact that the sequence $\frac{1}{n}\eta_{\xi(n)}$ converges to $\nu_A$ ``from all possible directions''. 
\end{rem}

\end{document}